\documentclass{article}
\title{An exact bandit model for the risk-volatility tradeoff}
\author{Max-Olivier Hongler \\EPFL/STI \\ \\Renaud Rivier\\ D\'epartement de math\'ematiques/Universit\'e de Gen\`eve.}
\usepackage{color}
\usepackage{latexsym}
\usepackage{amsmath}
\usepackage{amsfonts}
\usepackage{amssymb}
\usepackage{graphicx}
\usepackage[exercise]{amsthm}
\usepackage{cite}

\usepackage{booktabs}
\usepackage{xcolor}
\usepackage{hyperref}
\hypersetup{
    colorlinks,
    linkcolor={blue!90!black},
    citecolor={green!90!black},
    urlcolor={red!90!black}
}

\newcommand*{\deq}{\mathrel{\vcenter{\baselineskip0.65ex \lineskiplimit0pt \hbox{.}\hbox{.}}}=}
\newcommand{\ee}{\mathrm{e}}
\newcommand{\dd}{\mathrm{d}}
\newcommand{\E}{\mathbb{E}}

\newtheorem{thm}{Theorem}[section]

\newtheorem{cor}[thm]{Corollary}

\newtheorem{lem}[thm]{Lemma}
\newtheorem{exemp}[thm]{Example}
\theoremstyle{remark}
\newtheorem{rmk}{Remark}
\theoremstyle{plain}

\newtheorem*{Definition}{Definition}

\newtheoremstyle{note}
 {3pt}
 {3pt}
 {}
 {}
 {\itshape}
 {:}
 {.5em}
 {}

\theoremstyle{note}

\newtheoremstyle{citing}
 {3pt}
 {3pt}
 {\itshape}
 {}
 {\bfseries}
 {.}
 {.5em}
 {\thmnote{#3}}

\theoremstyle{citing}

\newtheoremstyle{break}
 {9pt}
 {9pt}
 {\itshape}
 {}
 {\bfseries}
 {.}
 {\newline}
 {}

\theoremstyle{break}

\theoremstyle{exercise}

\swapnumbers
\theoremstyle{plain}

\let\lvert=|\let\rvert=|

\begin{document}

\maketitle

\abstract{$\,$

\noindent
We revisit the two-armed bandit (TAB) problem where both arms are driven by diffusive stochastic processes with a common instantaneous reward. We focus on situations where the Radon-Nikodym derivative between the transition probability densities of the first arm with respect to the second is explicitly known. We calculate how the corresponding Gittins' indices behave under such a change of probability measure. This general framework is used to solve the optimal allocation of a TAB problem where the first arm is driven by a pure Brownian motion and the second is driven by a centered super-diffusive non-Gaussian process with variance quadratically growing in time. 
The probability spread due to the super-diffusion introduces an extra risk into the allocation problem. This drastically affects the optimal decision rule.
Our modeling illustrates the interplay between the notions of risk and volatility.
}

 \vspace{0.3cm}

 \vspace{0.3cm}
\noindent {\bf MSC2020 classification code:} 60G40, 93E20, 60J60.

\section{Basic problem, motivation and results}

\subsection{Introduction}

Assume you have to decide where to invest among a couple of assets. The assets are represented by stochastic processes. Both processes are zero mean preserving but differ in their volatility and their risk. A higher risk offers the prospect of higher gains but also comes with potentially higher losses. What should the optimal decision be? A heuristic reflection suggests that the answer should depend both on the asset reward structure and on your present wealth; indeed higher on-hand wealth is likely to weaken risk aversion and conversely. In the sequel, we address this question in a stylized manner by relying on the well-known multi-armed bandit formalism.

\noindent
Originally introduced by Robbins \cite{robbins1952some}, bandit algorithms are a central topic of study in mathematical optimization, with useful applications in various fields. 
These algorithms deal with problems where a fixed, limited set of actions, also called {\it arms}, are available, and the goal is to find an optimal strategy to select actions that maximize the cumulative reward. 
A certain level of uncertainty is usually assumed in the model. In some cases, the dynamics of the arms are unknown: a gambler can choose between different slot machines in a casino with no knowledge of the success probability of each machine. If the gambler uncovers the hidden probabilities, they can easily design an optimal strategy by always playing the highest probability arm. Such models have found applications have found numerous applications from the foundational paper of Thompson \cite{thompson2013bandit} about clinical trials, to economics (see \cite{bergemann2006bandit} for a revue) with a particular affinity with internet applications such as recommendation systems \cite{white2013bandit}, AB testing \cite{burtini2015survey} and dynamic pricing \cite{den2015surveys}.
As they provide a concise mathematical formulation of the exploration-exploitation dilemma in the context of Markov decision processes, bandit algorithms have recently benefited from the rise of reinforcement learning (see \cite[Part 1]{sutton2018reinforcement} and \cite[Section 1.1.2]{lattimore2020bandit} for explanations about the similarities and differences).

\noindent
We are here interested in the special case where the arms are real-valued Markovian processes with continuous trajectories. Both the dynamics of the arms, defined by some known Stochastic Differential Equation (SDE) and the reward function, defined as a non-decreasing continuous function, are known to the gambler. 
When an arm is engaged, it evolves according to some known Markovian dynamics and remains frozen otherwise.
The gambler's objective is to construct a real-time betting policy that specifies which arm to engage to maximize her total expected reward which is obtained as the accumulation of all instantaneous rewards.
This type of sequential decision problem is known as the multi-armed bandits (MAB) problem. Pioneered by \cite{W}, the MAB problem has a half-century-long record of publications \cite{BF,GJ }.
Several contributions \cite{M,MM,D} focus on finding optimal allocation policies for MAB with continuous time-diffusive processes.
Gittins showed that one can construct an optimal allocation policy by computing for every arm a deterministic function known as {\it  Gittins' indices} and always engaging the arm with the highest index.
The intuitive interpretation of the Gittins' index for one arm is the smallest instantaneous reward that makes immediate stopping profitable if the gambler's only option is to engage that arm or stop forever. 
As a result, the more expensive it is to convince a gambler to stop playing some game, the more profitable is expected by pursuing the gamble.

\noindent
In their seminal paper, Stiglitz and Rotshild \cite{RS} showed that the notion of volatility can take two different meanings: they distinguish between ordinary volatility (i.e.\ the diffusive coefficient) and the volatility spread relevant to model extra risky situations.  
It is natural to wonder how these two different notions affect the allocation decisions in a two-arm bandit and in particular how they enter into the corresponding Gittins' indices.
In other words, we ask the following question: how should a gambler play if an extra risk is introduced into one of the arms? To answer this question rigorously, we introduce the concept of Dynamic Mean-Preserving Spread (DMPS), distinguishing noisy processes and risky processes.

\noindent
In his famous paper, Karatzas \cite{K} developed the general method to compute Gittins' indices for Markov diffusion processes and derived the explicit expression for the drifted Brownian motion. 
In general, however, explicit expressions for the Gittins' index are unavailable hence the corresponding optimal allocation rules for general TAB. 
We will unveil a new possibility that allows us to compute explicit Gittins' indices for classes of diffusion processes to which the DMPS belongs. In particular, this will allow us to understand the optimal allocation policy for a TAB problem where one arm is a pure Brownian motion (volatile) and the other is a DMPS (risky). 
Ultimately, we pose and answer novel questions regarding the relationship between volatility and risk in the context of bandit algorithms.

\subsection{Model}
\noindent To address the question laid out in the previous section, we will specifically focus on the two-armed bandit problem with arms driven by the SDEs:
\begin{subequations}
  \begin{align}
    \dd X_t^{(1)} &= \sigma_1 \dd W^{(1)}_{t},\label{SDE-A}\\
    \dd X_{ t}^{(2)} &= \left[ \sigma_{2} \sqrt{2 \Gamma} \tanh \left( \frac{\sqrt{2 \Gamma}}{\sigma_2} X_{t}^{(2)}\right) \right] \dd t + \sigma_2 \dd W^{(2)}_{t}, \label{SDE-DMPS}
  \end{align}
\end{subequations}
\noindent with initial conditions $X^{(1)}_0 = X_{ 0}^{(2)} =0$
and where $(W_t^{(1)})_{t\geq 0}$ and $(W_t^{(2)})_{t\geq 0}$ are independent Wiener processes. 
Both arms are continuous stochastic processes with an additive {\it noise}: we call the parameters $\sigma_1, \sigma_2 > 0$ the variances.
The {\it spread}  $\Gamma \in \mathbb{R}^{+}$ models the extra risk by a probability spread. In line with \cite{ARH, RS}, this spread stylizes a risk increase for the arm $X_{t}^{(2)}$ compared to the pure Brownian motion $X_t^{(1)}.$ 
The corresponding transition probability densities (TPD) characterizing the Markovian diffusion processes $X^{(1)}_{t}$ and $X^{(2)}_{t}$ are well known \cite{ARH}:

\begin{subequations}
  \begin{align}
    P^{(1)}(x,t|0,0) &= \frac{1}{\sqrt{2 \pi \sigma_{1} ^{2} t}} \ee^{- \frac{x^{2} }{2\sigma_{1}^{2}t} } ,\label{TPD-BM} \\
    P^{(2)}(x,t|0,0) &= \frac{1}{2\sqrt{2 \pi \sigma_{2}^{2} t}} \left[ \ee^{- \frac{(x- \sigma_{2} ^{2} \sqrt{2\Gamma }t)^{2} }{2\sigma_{2}^{2}t} }  + \ee^{- \frac{(x+ \sigma_{2}^{2} \sqrt{2\Gamma }t)^{2} }{2\sigma_{2}^{2}t} } \right].
    \label{TPD-DMPS}
  \end{align}
\end{subequations}

\noindent The implications of Eq.(\ref{TPD-DMPS}) are fully discussed in \cite{ARH}. In particular, the first moments of Eqs.(\ref{TPD-BM}) and (\ref{TPD-DMPS}) read:

\begin{equation*}
  \begin{array}{l}
  \mathbb{E} \left[ X_t^{(1)} \right] = 0\qquad \,\, \qquad{\rm and} \qquad \qquad \mathbb{E} \left[ X_t^{(2)} \right] =0, \\\\ 
  \mathbb{E} \left[ \left( X_t^{(1)} \right)^{2} \right] = \sigma_{1}^{2} t \qquad{\rm and } \qquad  \mathbb{E} \left[ \left(X_t^{(2)} \right)^{2} \right] = \sigma_{2}^{2} \left[ t + 2 \sigma_{2}^{2}\Gamma t^{2}\right].
  \end{array}
\end{equation*}

\noindent The super-diffusive behavior of $X^{(2)}$ (i.e.\ quadratic growth of the variance) mirrors its DMPS characteristics.
Despite the non-Gaussian character of $X_t^{(2)}$, the corresponding Gittins' index can be derived and so will the exact allocation policy for the MAB problem. Compared to the pure Brownian motion's case, we shall observe that the extra risk strongly changes the optimal policy. 

\noindent Finally, let us note that the SDE Eq.\eqref{SDE-DMPS} appears in the literature in other contexts. In \cite{BL}, the authors gave a short and easy proof of its universal nature in terms of Brownian bridges. It is also studied in \cite[Example 2]{rogers1981markov}.

\begin{figure}[h]
  \begin{center}
    \def\svgwidth{0.9\columnwidth} 
    \small
\begingroup%
  \makeatletter%
  \providecommand\color[2][]{%
    \errmessage{(Inkscape) Color is used for the text in Inkscape, but the package 'color.sty' is not loaded}%
    \renewcommand\color[2][]{}%
  }%
  \providecommand\transparent[1]{%
    \errmessage{(Inkscape) Transparency is used (non-zero) for the text in Inkscape, but the package 'transparent.sty' is not loaded}%
    \renewcommand\transparent[1]{}%
  }%
  \providecommand\rotatebox[2]{#2}%
  \newcommand*\fsize{\dimexpr\f@size pt\relax}%
  \newcommand*\lineheight[1]{\fontsize{\fsize}{#1\fsize}\selectfont}%
  \ifx\svgwidth\undefined%
    \setlength{\unitlength}{422.34780139bp}%
    \ifx\svgscale\undefined%
      \relax%
    \else%
      \setlength{\unitlength}{\unitlength * \real{\svgscale}}%
    \fi%
  \else%
    \setlength{\unitlength}{\svgwidth}%
  \fi%
  \global\let\svgwidth\undefined%
  \global\let\svgscale\undefined%
  \makeatother%
  \begin{picture}(1,0.43927312)%
    \lineheight{1}%
    \setlength\tabcolsep{0pt}%
    \put(0,0){\includegraphics[width=\unitlength,page=1]{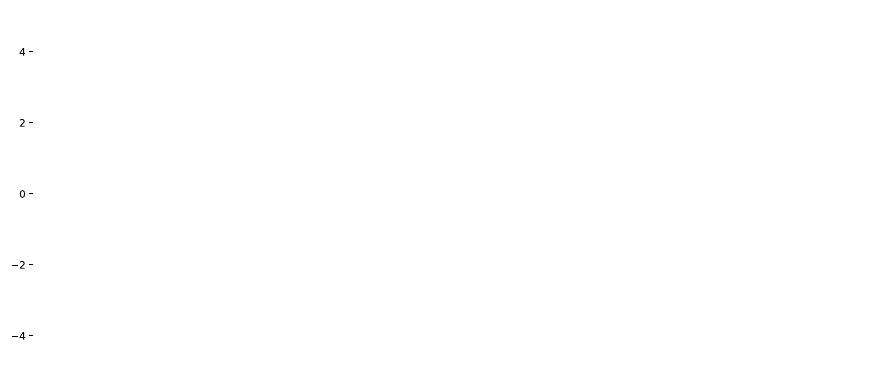}}%
    \put(0.00899457,0.17785657){\color[rgb]{0,0,0}\rotatebox{90}{\makebox(0,0)[lt]{\lineheight{1.14999998}\smash{\begin{tabular}[t]{l}$X^{(2)}_t$\end{tabular}}}}}%
    \put(0.51423341,0.19375907){\color[rgb]{0,0,0}\rotatebox{90}{\makebox(0,0)[lt]{\lineheight{1.14999998}\smash{\begin{tabular}[t]{l}$X^{(1)}_t$\end{tabular}}}}}%
    \put(0,0){\includegraphics[width=\unitlength,page=2]{sample_paths_inkscape.pdf}}%
    \put(0.03332599,0.43027854){\color[rgb]{0,0,0}\makebox(0,0)[lt]{\lineheight{1.14999998}\smash{\begin{tabular}[t]{l}Dynamic Mean Preserving Spread\end{tabular}}}}%
    \put(0,0){\includegraphics[width=\unitlength,page=3]{sample_paths_inkscape.pdf}}%
    \put(0.547786,0.4291311){\color[rgb]{0,0,0}\makebox(0,0)[lt]{\lineheight{1.14999998}\smash{\begin{tabular}[t]{l}Standard Brownian motion\end{tabular}}}}%
  \end{picture}%
\endgroup%
  \end{center}
  \caption{Illustration of sample paths: The red trajectories represent standard Brownian motion (BM), highlighting the stochastic nature of their evolution. For the DMPS, with parameter $\Gamma=0.2$, the {\it escape} from a certain threshold may occur at varying times. Once a certain critical level is surpassed, however, the dynamics of the system precipitate an accelerated divergence from the origin.}
  \label{fig:trajectories}
\end{figure}

\subsection{Results}
In a nutshell, our present address can be summarised as follows:

\begin{itemize}
 \item [a)]  {\bf Gittins' indices transformation under changes of probability measures}. One considers a couple of SDEs:
\begin{equation}
\label{TWO-SDE}
\left\{
\begin{array}{l}
\dd X^{(1)}_t = \mu(X_t) \dd t + \sigma \dd W_t,  \\\\

\dd X^{(2)}_{t} = \mu \left( X^{(2)}_{ t} \right) \dd t  + \left\{\sigma^{2}[\partial_{x} \ln F_{\Gamma} (x) ] \right\}_{x= X^{(2)}_{t}} \dd t + \sigma \dd W_t,
\end{array}
\right.
\end{equation}
\noindent where $\sigma, \Gamma > 0$ are positive constants, $(W_t)_{t \geq 0}$ is a Wiener process and $F_{\Gamma}(x) \geq 0$ solves the Ordinary Differential Equation (ODE):
\begin{equation*}
 \frac{\sigma^{2}}{2} \partial_{xx} F_{\Gamma}(x) + \mu(x) \partial_x F_{\Gamma}(x)  -\Gamma F_{\Gamma}(x) =0.
\end{equation*}
\noindent The Radon-Nikodym derivative provides a way to relate the probability measures given by the TPDs associated to $X^{(i)}$, $i=1,2.$ In particular, we know (see for example \cite[Theorem 2.1]{DAP}):  
\begin{equation*}
 \nu_{x_0}(x, t) \deq \frac{P^{(2)}(x,t \vert x_0, 0)}{P^{(1)}(x,t \vert x_0, 0)}
= \ee^{-\Gamma t} \frac{F_{\Gamma}(x)}{F_{\Gamma}(x_0)}.
\end{equation*}
 
\noindent Hence $P^{(2)}$ derives from  $P^{(1)}$ by the change of probability measure $\nu_{x_0}$ and it is natural to ask how this reflects for the corresponding Gittins' indices (see Theorem \ref{THM:GITTINS-DOOBS} below).

 \item [b)]  {\bf Optimal allocation for a class of non-Gaussian TAB}. The special case obtained from a) with $\mu(x)=0$  and $F_{\Gamma}(x) = \cosh\left[ \frac{\sqrt{2 \Gamma} }{\sigma} x\right]$ in Eq.\eqref{TWO-SDE} corresponds to the TAB defined by Eqs.(\ref{SDE-A}) and (\ref{SDE-DMPS}). 
 We introduce the Gittins' indices difference:
 $$ \Delta (x) \deq \Delta_{\Gamma, \sigma_1, \sigma_2}(x) = M^{(2)}(x) - M^{(1)}(x),$$
 where $M^{(i)}$, $i=1,2$, denotes the Gittins' index of the process $X^{(i)}$ (see Eq.\eqref{GIDEF} below).
Accordingly at location $x$,  it will be optimal to engage the arm $X_t^{(2)}$ if $ \Delta(x) >0$ and vice-versa. 
For $\Gamma >0$, the extra risk introduced into the $X_t^{(2)}$-evolution gives rise to a $x$-dependent sign (see Theorem \ref{OPTALOC}). This is in sharp contrast to the case $\Gamma =0$ where the difference $\Delta$ always carries the same sign.
  
\item [c)]  {\bf Risk spread versus volatility tradeoff}. As a consequence of the analysis of $\Delta$, the optimal strategy undergoes some phase transitions. In the case where both arms are pure diffusions ($\Gamma=0$), there are only two regions in the parameter space: $\sigma_1>\sigma_2$ and $\sigma_1 < \sigma_2$.
As we add risk in the second arm ($\Gamma >0$), a third region in the parameter space emerges (see  Figure \ref{FIGURE1}) where the optimal allocation results not only by comparing the volatilities of both arms. In this intermediate region, the optimal allocation depends on the state of the process. In positions yielding high instantaneous rewards, it becomes preferable to engage the risky arm even if $\sigma_2<\sigma_1$ and conversely. This interplay between volatility and risk (or variance and spread) is to the best of our knowledge new in the context of bandits algorithms.    
\end{itemize}

\noindent The content of the paper is organized as follows: in section \ref{SEC3}, we recall basic notations and basic results for multi-armed bandits. In section \ref{SEC4}, we derive the GI's for arm dynamics obtained by change of probability measures. Section \ref{SEC55} ends the paper by showing the optimal allocation policy for the two-bandit problem with arm's dynamics Eqs.(\ref{SDE-A}) and (\ref{SDE-DMPS}).

\section{Gittins' index for Markov diffusion processes} \label{SEC3}
 
\noindent We start by briefly recalling the formalism of the diffusive multi-armed bandit problem by adopting the notation of I. Karatzas \cite{K}.

\begin{Definition}{\it(Multi-armed diffusive bandit problem)}

\noindent Let $(\Omega, {\Sigma}, \mathbb{P})$ be a probability space endowed with the increasing family of $\sigma$-fields $({\mathcal F}_t: t \geq 0)$. The MAB problem is defined by the following set of elements:

\begin{itemize}

\item[1.] An admissible {\it reward} structure is a 4-tuple $(h, \alpha, k, K)$, where $\alpha > 0$ is called the discounting factor and $h : \mathbb{R} \rightarrow [\alpha k, \alpha K]$ is a strictly increasing function, with bounded first and second derivatives and $$
\lim_{x \rightarrow -\infty} h(x) = \alpha k, 
\quad 
\lim_{x \rightarrow +\infty} h(x) = \alpha K, 
\quad \lim_{\vert x \vert \rightarrow \infty} \vert h'(x)\vert = 0.
$$  

\item[2.]  A collection of {\it $N$ bandit's arms} consisting of $N$ stochastic independent diffusion processes:

\begin{equation*}
\left\{
\begin{array}{l}
\dd X_t^{(j)} =  \mu^{(j)}\bigl( X_t^{(j)}\bigr)  \dd t + \sigma^{(j)}\bigl( X_t^{(j)}\bigr) \dd W_t^{(j)}  , \qquad j=1,2,\cdots , N, \\\\

X_0^{(j)} = 0, 
\end{array}
\right.
\end{equation*}

\noindent where $\mathbf{W}_t \deq \left( W_t^{(1)}, \cdots, W_t^{(2)}\right)$ is a $N$-dimensional Brownian motion adapted to the filtration $(\mathcal F_t : t \geq 0)$,  
$\mu^{(j)} \in C^1(\mathbb{R}, \mathbb{R})$ and $\sigma^{(j)} \in C^1(\mathbb{R}, \mathbb{R}_{>0})$ with $\vert \partial_x\sigma^{(j)}(x) \vert$ bounded. 
The state of the system at time $t$ is written ${\bf X}_t \deq \left( X_t^{(1)}, \cdots, X_t^{(N)}\right)$.

\item[3.]  A set ${\cal A}$ of admissible {\it allocation policies} formed by  progressively measurable single-valued processes $A: \mathbb{R}^{+} \rightarrow \left\{ 1,2 \cdots, N\right\} $. Only one arm is engaged at any time $t$ and the evolution of all disengaged arms remains frozen.

\item[4.]   Under a specific allocation policy $A \in {\cal A}$, the {\it expected cumulative reward} ${\cal J}_{SDE-A}$ collected during an infinite time horizon is given by: 

\begin{equation*}
{\cal J}_{SDE-A} \deq \mathbb{ E} \left\{\int_{0}^{\infty} \ee^{-\alpha s} h\left( X_s^{(A(s))} \right)\dd s \right\},
\end{equation*}

\noindent where the expectation $\mathbb{E} \left\{ \cdot \right\} $ is taken over all the intermittent trajectories realized by the sequentially selected arms under policy $A(t)$.

\end{itemize}
\end{Definition} 

\noindent {\bf Optimal allocation policy $A^{*}$}. 
The {\it optimal allocation policy $A^{*}$} of a MAB is defined by the equality
\begin{equation*}
\displaystyle {\cal J}_{A^{*}} \deq \sup_{ A \in {\cal A} }{\cal J}_{SDE-A} .
\end{equation*}

\noindent
Let $(X_t)_{t \geq 0}$ be a real-valued stochastic process and $(h, \alpha, C, c)$ an admissible reward structure. 
In \cite[Theorem 4.1]{K}, the author defines the Gittins' index of $X$ as 
\begin{align}\label{GIDEF}
   M_{\alpha, X}(x) \deq 
   \sup_{\tau} \frac{\E_x \int_{0}^{\tau} h(X_s) \ee^{-\alpha s} \dd s}{1 - \E_x \ee^{-\alpha \tau}}, 
   \quad 
   x \in \mathbb{R}, 
\end{align} 
where the supremum is taken over all stopping times measurable with respect $\mathcal F_t$. 

\noindent It is established that the optimal allocation policy $A^{*}(t)$ is realized by systematically engaging the arm with the larger GI.
More precisely, if at time $t \geq 0$ the MAB state is ${\bf X}_t = (x_1, x_2, \cdots, x_N)$, then 
\begin{equation*} 
  A^*(t) = \operatorname*{argsup}_i  M_{\alpha, X^{(i)}} (x_i), 
\end{equation*} 
defines an optimal policy (equality is resolved by random tie break). 

\noindent In this article, we will focus on computing Gittins' indices and thus will mostly look at bandits in isolation one by one. In Section \ref{SEC55}, we compare the indices of two different processes to describe an optimal decision rule.

\subsection{Gittins' index for diffusive arm's evolution} 
 
\noindent Consider a single arm evolution :

\begin{equation}
  \label{GENERIC-SDE}
  \dd X_t = \mu(X_t) \dd t + \sigma \dd W_t, \quad X_0 =0,
\end{equation}

\noindent where $\mu$ is a drift function, $\sigma > 0$ is a positive constant \footnote{We limit our discussion to constant diffusion coefficient $\sigma(x)=\sigma$. For scalar processes, one always recovers this situation by introducing an ad-hoc change of variables (known as the Lamperti transform).}and $(W_t)_{t \geq 0} $ is a standard Wiener process. 
The TPD $P(x,t|x_0,0)$ is a solution of the forward Fokker-Planck equation:

\begin{equation}
\label{FOKP}
\left\{
\begin{array}{l}
\partial_{t} P(x,t|x_0,0) = {\mathcal F} [P(x,t|x_0,0)], \quad P(x,0|x_0,0)= \delta(x-x_0), \\\\
{\mathcal F} [\cdot] = -\partial_x \left\{\mu(x) [\cdot] \right\} + \frac{\sigma^{2}}{2} \partial_{xx} [\cdot],
\end{array}
\right.
\end{equation}
  
\noindent where $\delta(x-x_0)$ is the Dirac probability mass. Together with Eqs.\eqref{GENERIC-SDE} and \eqref{FOKP}, we further introduce the infinitesimal generator:

\begin{equation}
\label{INF-GENERATOR}
{\cal L}_{\alpha} [\cdot] =  \frac{\sigma^{2}}{2}\partial_{xx} [\cdot] + \mu(x) \partial_x [\cdot] - \alpha [\cdot] , 
\quad \alpha >0.
\end{equation}
\noindent
For $f,g \in C^1(\mathbb{R})$, we define the {\it Wronskian determinant} as 
$$
W[f, g](x) \deq {\rm Det}\begin{pmatrix}
    f(x)& g(x)\\
      f'(x) &   g'(x)
\end{pmatrix},
\quad x\in \mathbb{R}.
$$

\noindent 
 
\begin{thm}[I. Karatzas \cite{K}]\label{KARA}
\noindent Consider the diffusion Eq.\eqref{GENERIC-SDE} together with an admissible reward structure $(h, \alpha, k, K)$. Suppose in addition that there exists some positive constants $\alpha_0, \beta>0$ such that  
\begin{align}\label{KARA-CONDITION}
  0 < \alpha_0 \leq \alpha-\mu'(x)\leq \beta, \qquad x \in \mathbb{R}.
\end{align}
Then the Gittins' index $M_{\alpha}$ of the process $(X_t)_{t \geq 0}$   reads:
 
\begin{equation}
\label{GITTINS-FORMULA-WRONSKIAN}
M_{\alpha} (x) = \frac{ p_{\alpha}'(x) \varphi_{\alpha}(x) - \varphi_{\alpha}'(x) p_{\alpha}(x) } {- \varphi_{\alpha}'(x) } 
= \frac{W[p_{\alpha}, \varphi_{\alpha}](x)}{W[\,1\,, \,\varphi_{\alpha}](x)} ,
\end{equation}

\noindent where the functions $\varphi_{\alpha}(x)$ and $p_{\alpha}(x)$ satisfy:

\begin{equation*}
{\cal L}_{\alpha}[\varphi_{\alpha}(x)]=0, \quad \quad \displaystyle\lim_{x \rightarrow +\infty} \varphi_{\alpha}(x) =0,
\quad 
{\cal L}_{\alpha}[p_{\alpha}(x)]= h(x).
\end{equation*}

 \end{thm}
 
\begin{proof}[Proof of Theorem \ref{KARA}]
 The entire proof is exposed in  \cite[Section 3]{K}.
 \end{proof}

\noindent We propose an alternative analytic form of the Gittins index in terms of solutions to Ordinary Differential Equations (ODEs) corresponding to the infinitesimal generator $\cal L_{\alpha}$, $\alpha>0.$  

\begin{cor}\label{COR:ALTERNATIVE-GTTINS}
The Gittins' index defined in Eq.(\ref{GITTINS-FORMULA-WRONSKIAN}) can alternatively be written as:
  
\begin{equation}
\label{ALTERF}
M_{\alpha}(x) =\frac{W[\varphi_{\alpha}, \eta_{\alpha}] (x)}{W[ \,\varphi_{\alpha},1] (x)} \int_{x}^{\infty} \frac{2h(s)\varphi_{\alpha} (s)\dd s}{\sigma^{2}W[\varphi_{\alpha}, \eta_{\alpha}] (s)} ,
\end{equation}
\noindent with 

\begin{equation}
\label{HOMOG}
\left\{
\begin{array}{l}
{\cal L}_{\alpha}[\varphi_{\alpha}(x)]= {\cal L}_{\alpha}[\eta_{\alpha}(x)]=0, \quad 
W[\varphi_{\alpha}, \eta_{\alpha}](x) \neq 0, \quad x \in \mathbb{R}, \\\\

 \displaystyle\lim_{x \rightarrow +\infty} \varphi_{\alpha}(x) =0 \qquad {\rm and} \qquad \displaystyle\lim_{x \rightarrow -\infty} \eta_{\alpha}(x) =0.
\end{array}\right.
\end{equation}

 \end{cor}

\noindent To establish Corollary \ref{COR:ALTERNATIVE-GTTINS}, we recall the following classical result:

\begin{lem}\label{LEM:FIRST-TECHNICAL}
\noindent The solution $p_{\alpha}$ of the inhomogeneous ODE:

\begin{equation*}
\left\{
\begin{array}{l}
{\cal L} _{\alpha} [p_{\alpha}(x) ]  = h(x),\\\\

{\cal L} _{\alpha} [\cdot] \deq \frac{\sigma^{2}}{2}\partial_{xx} [\cdot] + \mu(x) \partial_{x} [ \cdot] - \alpha [\cdot], \quad x \in\mathbb{R},

\end{array}
\right.
\end{equation*}

\noindent can be written as:

\begin{equation}
\label{EQ:PALPHA-INTEGRAL}
p_{\alpha}(x) = - \varphi_{\alpha}(x) \int_{-\infty}^{x} \frac{2 h(s) \eta_{\alpha}(s)\dd s}{\sigma^{2} W[\varphi_{\alpha}, \eta_{\alpha}] (s)} +
\eta_{\alpha}(x) \int_{x}^{\infty} \frac{2 h(s) \varphi_{\alpha}(s)\dd s}{\sigma^{2} W[\varphi_{\alpha}, \eta_{\alpha}] (s)}, 
\end{equation}

\noindent where $\varphi_{\alpha}(x)$ and $\eta_{\alpha}(x)$ satisfy the homogeneous Eq.(\ref{HOMOG}).

\end{lem} 


\begin{proof}
\noindent This is a classical result for inhomogeneous second-order linear ODE. Define:

$$
\Pi_1(x) \deq \int_{-\infty}^{x} \frac{2 h(s) \eta_{\alpha} (s)\dd s }{\sigma^{2} W [ \varphi_{\alpha}, \eta _{\alpha}] (s)} 
\quad {\rm and} \quad 
\Pi_2(x) \deq \int_{x}^{\infty} \frac{2 h(s) \varphi_{\alpha} (s)\dd s }{\sigma^{2} W [ \varphi_{\alpha}, \eta _{\alpha}] (s)} ,
$$

\noindent and by direct substitution, one easily verifies that:
\begin{equation}
\label{PALF}
 {\cal L}_{\alpha}[-\varphi_{\alpha} (x) \Pi_1(x) + \eta_{\alpha} (x) \Pi_2(x) ] = {\cal L}_{\alpha} [p_{\alpha}(x) ] =h(x) , \\\\
\end{equation}
\end{proof} 

\begin{proof}[Proof of Corollary \ref{COR:ALTERNATIVE-GTTINS}]
Using Eq.\eqref{EQ:PALPHA-INTEGRAL}, we rewrite $p_{\alpha}$ in Eq.(\ref{PALF}). Plugging the later into Eq.(\ref{GITTINS-FORMULA-WRONSKIAN}), we obtain Eq.(\ref{ALTERF}).
\end{proof}

\begin{rmk}
\noindent Note that Eq.(\ref{ALTERF}) also appears in \cite[Eq.(3.17)]{K}. 
\end{rmk}

\begin{exemp}[Gittins' index for the Brownian motion]
 For the Brownian motion Eq.(\ref{SDE-A}), we have:
 
 $$
 \left\{
\begin{array}{l}
 {\cal L}_{\alpha} [\cdot] = \frac{\sigma_1^{2}}{2} \partial_{xx} [\cdot] -\alpha[\cdot], \\\\
 \varphi_{\alpha}(x) = \ee^{- \frac{\sqrt{2 \alpha}}{\sigma_1} x}\quad {\rm and}\quad \eta_{\alpha}(x) = \ee^{ \frac{\sqrt{2 \alpha}}{\sigma_1} x},
 \\\\
W[\varphi_{\alpha}, \eta_{\alpha}](x) = \frac{2 \sqrt{2\alpha}}{\sigma_1} \quad {\rm and}\quad W[ \varphi_{\alpha},1](x) = \frac{\sqrt{2\alpha} }{\sigma_1} \ee^{- \frac{\sqrt{2 \alpha}}{\sigma_1}x} .

 \end{array}
 \right.
 $$

\noindent Accordingly,  from Eq.(\ref{ALTERF}) we have :
\begin{equation}
\label{PUREBrownian motion}
M_{\alpha}(x) = 
\frac{\sigma_1}{\sqrt{2 \alpha}} \int_{x}^{\infty} \frac{2h(s)}{\sigma_{1}^{2} } \ee^{- \frac{\sqrt{2\alpha}}{\sigma_1} [ s-x]}\dd s=
 \frac{1}{\alpha}\int_{0}^{\infty} h \left[ x + \frac{\sigma_1 z}{\sqrt{2 \alpha}} \right] \ee^{-z} \dd z, \end{equation}
 
\noindent in agreement with \cite[Eq.(3.22)]{K}.
\end{exemp}
\section{Gittins' indices under change of probability measures} \label {SEC4}

\noindent In this section, we look at how the Gittins' index can be computed using a change of measure. 

\noindent Let us consider the scalar diffusion introduced in Eq.(\ref{GENERIC-SDE}) with infinitesimal generator $\mathcal{L}_{\alpha}$, $\alpha>0$,  defined in Eq.(\ref{INF-GENERATOR}), and homogeneous solutions $\varphi_{\alpha}$ and $\eta_{\alpha}$ satisfying Eq.\eqref{HOMOG}.  
For a set of real constants $p, q \in \mathbb{R}$ and $\Gamma > 0$, assume that we have a positive definite function $F_{\Gamma}:\mathbb{R}\rightarrow\mathbb{R}_{>0}$ which satisfies: 

\begin{equation}
\label{FDEF}
F_{\Gamma}(x) \deq p \varphi_{\Gamma}(x) + q \eta_{\Gamma}(x) >0,
\quad x \in \mathbb{R},
\end{equation}
where $\varphi_{\Gamma}$ and $\eta_{\Gamma}$ satisfy Eq.\eqref{HOMOG} for the generator $\mathcal{L}_{\Gamma}.$  
 Our main result reads:

\begin{thm}[Gittins' index and change of probability measure] \label{THM:GITTINS-DOOBS}
Let $\Gamma, \alpha>0$ be positive constants and $(h, \alpha, k, K)$ an admissible reward structure. Let $(X_t)_{t \geq 0}$ be the solution of the SDE:

 \begin{equation}
\label{MODIFFU}
 \dd X_{t} = m(X_t) \dd t + \sigma \dd W_t, 
 \quad 
 m(x) \deq \mu(x)+\sigma^2\partial_x\left\{\ln F_{\Gamma}(x)\right\},
\end{equation}

\noindent where the drift $m$ satisfies
\begin{align*}
   0<\alpha_0<\alpha - m'(x)<b,
\end{align*}
for some positive constants $\alpha_0,b>0.$ 
Then the corresponding Gittins' index $M_{\alpha, \Gamma}(x)$  reads :

\begin{equation}
\label{GITTINS-WRONSK-DOOBS}
M_{\alpha, \Gamma} (x) = \frac{W[\varphi_{\alpha+ \Gamma}, \eta_{\alpha+ \Gamma}] (x)}{W[ \varphi_{\alpha+ \Gamma}, F_{\Gamma}] (x)} \int_{x}^{\infty} \frac{2h(s)\varphi_{\alpha+ \Gamma} (s) F_{\Gamma}(s)\dd s}{\sigma^{2}W[\varphi_{\alpha+ \Gamma}, \eta_{\alpha+ \Gamma}] (s)},
\end{equation}

\noindent where $\varphi_{\alpha+ \Gamma}$ and $\eta_{\alpha+ \Gamma}$ satisfy Eq.\eqref{HOMOG} for the generator $\mathcal{L}_{\alpha+\Gamma}.$  

\end{thm}

\noindent The proof of Theorem \ref{THM:GITTINS-DOOBS} will follow from two intermediate results.


\begin{lem} \label{LEM:FIRST-TECHNICAL-DOOBS}
Let $\Gamma, \alpha >0$ be positive constants, $F:\mathbb{R}\rightarrow \mathbb{R}_{>0}$ be as in Eq.\eqref{FDEF} and 
\begin{equation*}
  {\cal L} _{\alpha, \Gamma} [\cdot] \deq \frac{\sigma^{2}}{2}\partial_{xx} [\cdot] + \left\{ \mu(x) + \sigma^{2}\partial_x \ln [F_{\Gamma}(x) ] \right\} \partial_{x} [ \cdot] - \alpha [\cdot],
\end{equation*}  
be an infinitesimal generator.
Let $\varphi_{\alpha, \Gamma}, \eta_{\alpha, \Gamma}: \mathbb{R} \rightarrow \mathbb{R}$ satisfy:

\begin{equation*}
\left\{
\begin{array}{l}
{\cal L} _{\alpha, \Gamma} [\varphi_{\alpha, \Gamma}(x) ]  = {\cal L} _{\alpha, \Gamma} [\eta_{\alpha, \Gamma}(x) ]= 0,
\quad W[\varphi_{\alpha, \Gamma}, \eta_{\alpha, \Gamma}](x) \neq 0, \quad x \in \mathbb{R}, \\\\
\displaystyle \lim_{x \rightarrow +\infty} \varphi_{\alpha, \Gamma}(x)=0\quad and \quad \displaystyle \lim_{x \rightarrow -\infty} \eta_{\alpha, \Gamma}(x)=0 
\end{array}
\right.
\end{equation*}

\noindent Then we have:

\begin{equation*}
\varphi_{\alpha, \Gamma}(x) = \frac{\varphi_{\alpha +\Gamma} (x) }{F_{\Gamma} (x)} \quad {\rm and } \quad \eta_{\alpha, \Gamma}(x) = \frac{\eta_{\alpha +\Gamma} (x) }{F_{\Gamma} (x)},
\quad x \in \mathbb{R},
\end{equation*}

\noindent where $\varphi_{\alpha+ \Gamma}$ and $\eta_{\alpha+ \Gamma}$ satisfy Eq.\eqref{HOMOG} for the generator $\mathcal{L}_{\alpha+\Gamma}.$  
\end{lem}

\begin{proof}[Proof of Lemma \ref{LEM:FIRST-TECHNICAL-DOOBS}] Dropping the $x$-arguments and writing $\partial_x f \deq f^{'}$ and $\partial_{xx} f \deq f^{''},$ we have:

$$
\left\{
\begin{array}{l}
 \left[ \frac{\varphi_{\alpha + \Gamma}}{F_{\Gamma}} \right]^{'}=\frac{ \varphi^{'}_{\alpha + \Gamma}}{F_{\Gamma}} - \frac{\varphi_{\alpha + \Gamma} F_{\Gamma}'}{F^{2}_{\Gamma}},
\\\\
\left[ \frac{\varphi_{\alpha + \Gamma}}{F_{\Gamma}} \right]^{''}= 
\frac{\varphi_{\alpha+ \Gamma}^{''}}{F_{\Gamma}} - 2 \frac{\varphi_{\alpha+ \Gamma}^{'} F_{\Gamma}^{'}}{F_{\Gamma}^{2}}- \frac{\varphi_{\alpha+ \Gamma} F_{\Gamma}^{''}}{F_{\Gamma}^{2}} + 2 \frac{\varphi_{\alpha + \Gamma} \left[F_{\Gamma}^{'} \right]^{2} }{F^{3}_{\Gamma}}.
\end{array}
\right.
$$

\noindent Accordingly, we obtain:

\begin{align*}
F_{\Gamma}& {\cal L}_{\alpha, \Gamma} \left[ \frac{ \varphi_{\alpha + \Gamma} }{F_{\Gamma}} \right]
= 
\frac{\sigma^{2}}{2} \left[ \varphi_{\alpha+ \Gamma}^{''} -
2 \frac{\varphi_{\alpha+ \Gamma}^{'} F_{\Gamma}^{'}}{F_{\Gamma}}- \frac{\varphi_{\alpha+ \Gamma} F_{\Gamma}^{''}}{F_{\Gamma} } 
2 \frac{\varphi_{\alpha + \Gamma} \left[F_{\Gamma}^{'} \right]^{2} }{F^{2}_{\Gamma}} \right]  \\
&+ \mu \left[ \varphi^{'}_{\alpha + \Gamma}- \frac{\varphi_{\alpha + \Gamma} F_{\Gamma}'}{F_{\Gamma}} \right] +
\frac{ \sigma^{2}F_{\Gamma}^{'} \varphi^{'}_{\alpha + \Gamma} }{F_{\Gamma} }
 - \sigma^{2}\left[ \frac{\varphi_{\alpha + \Gamma} \left(F_{\Gamma}^{'} \right)^{2}}{F_{\Gamma}^{3}} \right]
 - \alpha \varphi_{\alpha + \Gamma}
\\
&= \sigma^{2} \varphi_{\alpha+ \Gamma}^{''} - \frac{\sigma^{2}}{2} \left[ \frac{ F_{\Gamma}^{''} \varphi_{\alpha+ \Gamma}}{F_{\Gamma}}\right] + \mu \varphi^{'}_{\alpha + \Gamma} 
- \mu \frac{F_{\Gamma}^{'}\varphi_{\alpha + \Gamma}}{F_{\Gamma}} 
- \alpha \varphi_{\alpha + \Gamma}  \\
&= \sigma^{2} \varphi_{\alpha+ \Gamma}^{''}  + \mu \varphi^{'}_{\alpha + \Gamma} - (\alpha + \Gamma) \varphi_{\alpha+ \Gamma}= {\cal L}_{\alpha + \Gamma} [\varphi_{\alpha+ \Gamma}]=0 .
\end{align*}

\noindent Since $F_{\Gamma} >0$ we deduce that ${\cal L}_{\alpha, \Gamma} \left[ \frac{ \varphi_{\alpha + \Gamma} }{F_{\Gamma}}\right] = 0 $ and we can conclude that $\varphi_{\alpha, \Gamma} = \frac{\varphi_{\alpha + \Gamma} }{F_{\Gamma}}$. A similar computation holds for $\eta_{\alpha+ \Gamma}(x)$. 
\end{proof}


\begin{lem}\label{LEM:SECOND-TECHNICAL-DOOBS}
Under the hypotheses of Lemma \ref{LEM:FIRST-TECHNICAL-DOOBS}, the following identities hold for every $x \in \mathbb{R},$ 
\begin{equation*}
\begin{array}{l}
i) \,\quad W[\varphi_{\alpha, \Gamma}, 1] (x) = \frac{W[ \varphi_{\alpha + \Gamma} , F_{\Gamma} ](x)}{F_{\Gamma}^{2}(x)} , \\\\ 

ii) \quad W[\varphi_{\alpha, \Gamma}, \eta_{\alpha, \Gamma}] (x) = \frac{W[ \varphi_{\alpha + \Gamma}, \eta_{\alpha + \Gamma} ] (x) }{F_{\Gamma}^{2}(x)} .
\end{array}
\end{equation*}

\end{lem}

\begin{proof}[Proof of Lemma \ref{LEM:SECOND-TECHNICAL-DOOBS}] Immediate by using Lemma \ref{LEM:FIRST-TECHNICAL-DOOBS} and direct calculations.
\end{proof}

\begin{proof}[Proof of Theorem \ref{THM:GITTINS-DOOBS}]
Invoking Corollary \ref{COR:ALTERNATIVE-GTTINS}, the Gittins' index $M_{\alpha, \Gamma}(x)$ for the diffusive evolution Eq.(\ref{MODIFFU}) reads immediately as:

\begin{equation}
\label{GITON22}
M_{\alpha, \Gamma}(x) = \frac{W[\varphi_{\alpha, \Gamma}, \eta_{\alpha, \Gamma}](x)}{W[ \varphi_{\alpha, \Gamma},1](x)}
 \int_{x}^{\infty} \frac{2 h(s) \varphi_{\alpha, \Gamma}(s)\dd s }{\sigma^{2} W[\varphi_{\alpha, \Gamma}, \eta_{\alpha, \Gamma}](s)} .
\end{equation}

\noindent Using Lemmas \ref{LEM:FIRST-TECHNICAL-DOOBS} and \ref{LEM:SECOND-TECHNICAL-DOOBS}, we then find that:

\begin{equation}
\label{INTERRES}
\left\{
\begin{array}{l}
\frac{W[\varphi_{\alpha, \Gamma}, \eta_{\alpha, \Gamma}](x)}{W[ \varphi_{\alpha, \Gamma},1](x)} = \frac{W[\varphi_{\alpha+ \Gamma}, \eta_{\alpha+ \Gamma}] (x)}{W[\varphi_{\alpha+ \Gamma}, F_{\Gamma}] (x)} ,
\\\\
\frac{2 h(s) \varphi_{\alpha, \Gamma}(s) }{\sigma^{2} W[\varphi_{\alpha, \Gamma} , \eta_{\alpha, \Gamma} ](s) } = 

\frac{2 h(s) \varphi_{\alpha + \Gamma} (s) F_{\Gamma}(s)}{\sigma^{2}W[\varphi_{\alpha + \Gamma}, \eta_{\alpha + \Gamma}] (s) }.
\end{array}
\right.
\end{equation}

\noindent Plugging Eq.(\ref{INTERRES}) into Eq.(\ref{GITON22}), the assertion follows immediately.
\end{proof}


\noindent Let us work out some consequences of Theorem \ref{THM:GITTINS-DOOBS}.
As a first illustration, we will derive the Gittins' index of a Brownian motion with a constant drift. This result can be found in \cite[Eq.(3.22)]{K} and our corollary can be seen as a sanity check.

\begin{cor}\label{CG4} 
\noindent Let $( h, \alpha, k, K)$ an admissible reward structure, $\mu \in \mathbb{R}$ and $ \sigma \in \mathbb{R}^{+}$ two constants. For the drifted Brownian motion diffusion process:

\begin{equation}
\label{GITTINS-BM-2}
\dd X_t = \mu \dd t + \sigma \dd W_t, \qquad X_0=0,
\end{equation}

\noindent the corresponding Gittins' index $M_{\alpha} (x) $ reads:

\begin{equation*}
M_{\alpha} (x) = \frac{1}{\alpha} \int_{0}^{\infty} h\left[x + \frac{z}{\beta} \right] \ee^{-z}dz, \qquad \quad \beta \deq \frac{\sqrt{\mu^{2} + 2 \alpha \sigma^{2}} - \mu }{\sigma^{2}}.
\end{equation*}

\end{cor}

\begin{proof} 
\noindent We start with the pure Brownian motion which is defined by the SDE $\dd X_t = \sigma \dd W_t$. The corresponding infinitesimal generator is:

\begin{equation} \label{VS1}
\left\{
\begin{array}{l}
{\cal L}_{\alpha} [ \cdot]= \frac{\sigma^{2}}{2} \partial_{xx} [\cdot] - \alpha [\cdot], \\\\
\varphi_{\alpha} (x) = \ee^{- \frac{\sqrt{2 \alpha}}{\sigma} x} \quad \text{and} \quad \eta_{\alpha} (x) = \ee^{+ \frac{\sqrt{2 \alpha}}{\sigma} x}.
\end{array}
\right.
\end{equation}

\noindent The functions $\varphi_{\alpha}  $ and $\eta_{\alpha}  $ satisfy the conditions of Eq.\eqref{HOMOG} and form a basis of the vector space of solutions of Eq.(\ref{VS1}). In Eq.(\ref{FDEF}), we make the specific choice $ p=0$ and $ q=1$ leading to:
 
\begin{equation*}
F_{\Gamma}(x) = \ee^{+ \frac{\sqrt{2 \Gamma}}{\sigma} x} \quad \Leftrightarrow \quad \sigma^{2}\partial_x \ln [ F_{\Gamma}(x)] = \sigma \sqrt{2 \Gamma}.
\end{equation*}
 
\noindent Writing $ A \deq  \frac{\sqrt{2 \Gamma}} {\sigma}$ and $B \deq \frac{ \sqrt{2 (\alpha + \Gamma)}}{\sigma}$, a direct computation shows that:

\begin{equation}
\label{ABDEF}
\left\{
\begin{array}{l}
W[\varphi_{\alpha+ \Gamma} , F_{\Gamma}  ](x)= 
 (B+A) \ee^{(B-A)x },
\\\\
W[\varphi_{\alpha+ \Gamma}, \eta_{\alpha+ \Gamma} ](x) = \frac{2 \sqrt{2 (\alpha + \Gamma)}}{\sigma} = 2B,
\end{array}
\right.
\end{equation}
for every $x \in \mathbb{R}.$

\noindent Fixing  $\Gamma= \frac{\mu^{2} }{2 \sigma^{2}}$, we see that $\sigma^{2}\partial_x \ln [ F_{\Gamma}(x)] = \mu $ and therefore Eq.\eqref{GITTINS-BM-2} can be modified in the format of Eq.\eqref{MODIFFU} with $m(x)=\sigma^{2}\partial_x \ln [ F_{\Gamma}(x)=\mu.$ 
\noindent Finally, observe that since $m$ is constant, condition Eq.\eqref{KARA-CONDITION} is trivially satisfied and we can apply Theorem \ref{THM:GITTINS-DOOBS}. The Gittins' index of $X_t$ can be read from  Eq.(\ref{GITTINS-WRONSK-DOOBS}):

\begin{equation*}
M_{\alpha, \Gamma} (x) 
= \frac{\ee^{(B-A) x}}{B+A}\int_{x}^{\infty} \frac{2 h(s)} {\sigma^{2}} \ee^{(A-B) s}\dd s 
= \frac{1}{\alpha} \int_{0}^{\infty} h\left[x + \frac{\sigma z}{A-B} \right] \ee^{-z} \dd z,
\end{equation*}
where we used the change of variable $s = x + \frac{\sigma z}{A-B}$ in the second equality. By our choice of $\Gamma$, we see that $\frac{A-B}{\sigma}=\beta.$ This concludes the proof.  
\end{proof}

\noindent We now compute the Gittins' index of a class of diffusion that encompasses the DMPS diffusion defined in Eq.\eqref{SDE-DMPS}.
\begin{cor}\label{CG5}
 \noindent Let $( h, \alpha, k, K)$ an admissible reward structure, $\Gamma\in \mathbb{R}^{+}$ and $\sigma \in \mathbb{R}^{+}$ two constants such that $\Gamma<\alpha/2.$ Let us consider the non-Gaussian diffusion process:
  
\begin{equation*}
  \dd X_t = \sigma^{2} A \tanh( AX_t) \dd t + \sigma \dd W_t, \qquad X_0=0.
\end{equation*}

\noindent The corresponding Gittins' index $M_{\alpha, \Gamma} (x) $ reads:

\begin{equation}
\label{GITTINS-DMPS}
\left\{
\begin{array}{l}
M_{\alpha, \Gamma} (x) =\rho_- (x) S_{-}(x) +\rho_+ (x) S_{+}(x),\\\\

\rho_\pm (x) \deq \frac{2 }{2 \alpha + \sigma^{2} \left[ B\pm A \right]^{2} \ee^{\pm 2A x} } \quad {\rm and} \quad S_{\pm }(x) \deq \int_{0}^{\infty} h\left[x + \frac{s}{ B \pm A } \right] \ee^{-s}\dd s,
\end{array}
\right.
\end{equation}
\noindent where both constants $A$ and $B$ are defined in Eq.(\ref{ABDEF}).
\end{cor}

 
\begin{proof}
Start with the pure Brownian motion and hence Eq.(\ref{VS1}). Make the specific choice $ p= q =\frac{1}{2}$ in  Eq.(\ref{FDEF}) to obtain :

\begin{align*}
F_{\Gamma} (x)= \cosh\left[\frac{\sqrt{2\Gamma}}{\sigma}x \right]= \cosh(Ax) \quad \Rightarrow \quad \sigma^{2} \partial_x \ln \left[ F_{\Gamma}(x) \right] = \sigma^{2} A \tanh \left( A x\right) .
\end{align*}
 
\noindent By direct calculations, we obtain:

$$
\left\{
\begin{array}{l}
\varphi_{\alpha + \Gamma} = \ee^{- Bx} \qquad \text{and} \qquad \eta_{\alpha + \Gamma} = \ee^{ Bx}, \\\\ 
W[\varphi_{\alpha + \Gamma}, \eta_{\alpha+ \Gamma}](x) = 2B,
\\\\
W[\varphi_{\alpha + \Gamma}(x), \cosh(Ax) ]= \frac{1}{2} \left[ (B+A)\ee^{(A-B)x} +(B-A) \ee^{- (B+A) x}\right].
\end{array}
\right.
$$

\noindent Before applying Theorem \ref{THM:GITTINS-DOOBS}, we still need to check Eq.\eqref{KARA-CONDITION}. Without loss of generality let us set $\sigma=1$ and write $m(x) = \frac{\dd}{\dd x}\ln F_{\Gamma}(x)$.  Then we find, using the fact that $\mathcal{L}_{\Gamma} F_{\Gamma}=0$, 
\begin{align}\label{equ:KARACONDI_DMPS}
  \begin{split}
  \alpha - \mu'(x) &=
  \alpha - \frac{F_{\Gamma}''(x)}{F_{\Gamma}(x)} - \Bigl(\frac{F_{\Gamma}'(x)}{F_{\Gamma}(x)}\Bigr)^2\\
  &=\alpha -  2\Gamma + 2 \frac{F_{\Gamma}'(x)}{F_{\Gamma}(x)}\\
  &= \alpha - 2\Gamma - 2\Gamma \tanh^2 \Bigl(\sqrt{2\Gamma} x\Bigr) \geq \alpha-2\Gamma.
  \end{split}
\end{align} 
The right-hand side is bounded away from $0$ if $\Gamma<\alpha/2.$  

\noindent We can therefore use Theorem \ref{THM:GITTINS-DOOBS}. Using Eq.(\ref{GITTINS-WRONSK-DOOBS}), we have:

\begin{align*} 
  M_{\alpha, \Gamma}(x) & = \frac{1}{W[\varphi_{\alpha + \Gamma}(x), \cosh(Ax) ]} \int_{x}^{\infty} \frac{2}{\sigma^{2}} h(s) \ee^{- Bs} \cosh(As)\dd s \\
  &=
 \frac{2 \ee^{Bx} }{ \left[ (B+A)\ee^{Ax} +(B-A) \ee^{- A x}\right]} \int_{x}^{\infty} \frac{2}{\sigma^{2}} h(s) \ee^{- Bs} \cosh(As)\dd s\\
 &= \frac{ \int_{x}^{\infty} \frac{2}{\sigma^{2}} h(s) \ee^{- (B-A) (s-x) }\dd s }{(B+A) + (B-A) \ee^{- 2Ax}}
+ \frac{ \int_{x}^{\infty} \frac{2}{\sigma^{2}} h(s)\ee^{- (B+A) (s-x) }\dd s} {(B+A) \ee^{+2Ax}+ (B-A)} \\
&=\frac{\int_{0}^{\infty}h\left[ x + \frac{z}{(B-A)}\right] \ee^{- z} \dd z}{2 \alpha + \sigma^{2} (A-B)^{2}\ee^{- 2 Ax} } +
\frac{\int_{0}^{\infty} h\left[ x + \frac{z}{(B+A)}\right] \ee^{- z} \dd z}{2 \alpha + \sigma^{2} (B+A)^{2}\ee^{ 2 Ax} }.
\end{align*}
\noindent Then a direct identification with Eq.(\ref{GITTINS-DMPS}) ends the proof.
\end{proof}

\section{Optimal allocation for the two-armed bandit} \label{SEC55}

\noindent We now discuss the optimal allocation of the two-armed non-Gaussian diffusive bandit with arms' dynamics given by  Eqs.(\ref{SDE-A}) and (\ref{SDE-DMPS}). 

\begin{thm}\label{OPTALOC}
\noindent Let $\Gamma, \sigma_1, \sigma_2 >0 $ be constant parameters, $( h, \alpha, k, K)$ be an admissible reward structure, and the two-armed diffusive evolutions:
\begin{equation*}
\left\{
\begin{array}{l}
\dd X^{(1)}_t = \sigma_1 \dd W_t^{(1)}, \\\\ 
\dd X^{(2)}_t = \sigma_2 \sqrt{2 \Gamma} \tanh\left[\frac{\sqrt{2 \Gamma}}{\sigma_2} X^{(2)}_t \right] \dd t +\sigma_2 \dd W_{2}^{(t)},
\end{array}
\right.
\end{equation*}

\noindent with corresponding Gittins' indices $M^{(1)} $ and $M^{(2)} $ given by Eq.(\ref{PUREBrownian motion}) respectively Eq.(\ref{GITTINS-DMPS}). Assume that the control parameters satisfy the relation:

\begin{equation}
\label{DIAGRA}
\sqrt{1 + \frac{\Gamma}{\alpha} }- \sqrt{ \frac{\Gamma}{\alpha} } < \frac{\sigma_2}{\sigma_1} 
< \sqrt{1 + \frac{\Gamma}{\alpha} } + \sqrt{ \frac{\Gamma}{\alpha} }
\end{equation}

\noindent Then there exists two thresholds $ x_{\pm} \deq x_{\pm} (\sigma_1, \sigma_2, \Gamma) \in \mathbb{R}$ such that:

$$
\left\{
\begin{array}{l}
 x > x_{+}(\sigma_1, \sigma_2, \Gamma) \quad \Rightarrow \quad M^{(1) }(x) < M^{(2)}(x), \\\\ 
 x <  x_{-}(\sigma_1, \sigma_2, \Gamma)  \quad \Rightarrow \quad M^{(1) }(x) > M^{(2)}(x).
 \end{array}
 \right.
$$
\end{thm}

\noindent
As will become clear from the proof, if the parameters do not satisfy Eq.\eqref{DIAGRA}, then the sign of the difference $\Delta$ remains constant over all $\mathbb{R}.$ 
If $\frac{\sigma_2}{\sigma_1}$ is smaller than the lower-bound, then $M^{(1)}(x)>M^{(2)}(x)$ for every $x \in \mathbb{R}$. Conversely, if $\frac{\sigma_2}{\sigma_1}$ lies above the upper-bound then $M^{(2)} (x) > M^{(1)}(x)$ for every $x \in \mathbb{R}.$ The various regimes can be summarized in the following phase diagram. 

\begin{figure}[ht!]
  \begin{center}
   \def\svgwidth{0.9\columnwidth} 
    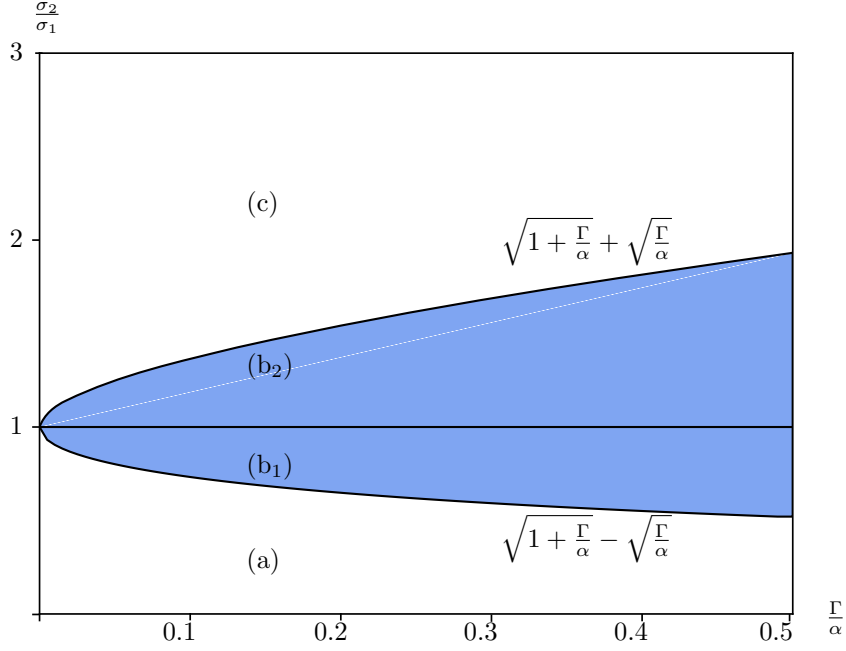
   \end{center}
    \caption{Phase diagram of Theorem \ref{OPTALOC}. In the regimes covered by region (a), the gambler should always favor the Brownian motion $X^{(1)}$. Conversely, in the regimes covered by region (c), the gambler should always favor the DMPS $X^{(2)}$. For instance, if $\Gamma=0$, both arms are Gaussian and the gambler should always favor $X^{(1)}$ over $X^{(2)}$ if $\sigma_1>\sigma_2$, and vice-versa (see Corollary \ref{CG4}). On the other hand, the regions ($\text{b}_1$) and ($\text{b}_2$) exhibit mixed behaviors: for instance in region ($\text{b}_1$), although the first arm $X^{(1)}$ has more volatility than the second arm $X^{(2)}$ (i.e.\  $\sigma_1>\sigma_2$), it is not true that the gambler should always favor the former. Indeed above some threshold value $x_+ \in \mathbb{R}$, representing a certain level of wealth, the DMPS is more favorable than the Brownian motion. In region ($\text{b}_2$), $X^{(2)}$ is more volatile than $X^{(1)}$: nevertheless, after seeing her wealth dropping below a certain threshold $x_- \in \mathbb{R}$, the gambler should favor the arm $X^{(1)}$.
    This mix of behaviors highlights the key difference between Gaussian diffusions, where one should always choose the more volatile arm, and the case of the non-Gaussian case, where the optimal allocation depends on the current state. Risk and volatility compete in a non-trivial way. }
    \label{FIGURE1}
 \end{figure}

\begin{proof}[Proof of Theorem \ref{OPTALOC}]
\noindent Define the indices difference:

\begin{equation*}
\Delta (x) \deq M^{(2) }(x)- M^{(1)} (x) = M_{\alpha}(x)- M_{\alpha, \Gamma}(x), 
\end{equation*}

\noindent with the Gittins' indices $M_{\alpha}(x)$ and $M_{\alpha, \Gamma}(x)$ respectively given by Eq.(\ref{PUREBrownian motion}) with $\sigma = \sigma_1$ and Eq.(\ref{GITTINS-DMPS}) with $\sigma = \sigma_2$. Completing the notations given in  Eqs.(\ref{ABDEF}) and Eq.(\ref{GITTINS-DMPS}), we write:

\begin{equation*}
\left\{
\begin{array}{l}
S_{0} \deq \int_{0}^{\infty} h\left[ x + \frac{z}{B_0}\right] \ee^{-z} \dd z,
\quad {\rm and } \quad 
S_{\pm} \deq \int_{0}^{\infty} h \left[ x \pm \frac{z}{B \pm A} \right] \dd z, \\\\

\rho_0\deq \frac{1}{\alpha}   \qquad {\rm and } \qquad  \rho_{\pm} \deq \frac{2}{2 \alpha + \sigma_{2}^{2} ( B\pm A)^{2} \ee^{\pm A x} },

\\\\

B_0 \deq \frac{\sqrt{2 \alpha}}{\sigma_1} , \qquad A\deq \frac{\sqrt{2\Gamma }}{\sigma_2} , \qquad {\rm and} \quad B\deq \frac{\sqrt{2(\alpha + \Gamma)} }{\sigma_2}.
\end{array}
\right.
\end{equation*}

\noindent By Eq.\eqref{DIAGRA}, we see that $B-A < B_0<B+A$. Moreover, the parameters $A$ and $B$ satisfy the identity: 
$$
\frac{\sigma_{2}^{2} (B-A)^{2}}{2 \alpha } = \frac{2 \alpha} {\sigma_{2}^{2} (B+A)^{2}}.
$$

\noindent Defining $y(x) \deq \frac{\sigma_{2}^{2} (A-B)^{2}}{2 \alpha} \ee^{-2Ax}$, we have:

\begin{equation*}
\rho_{-}(x) \deq \frac{\rho_{0}}{1+y(x)} \quad {\rm and} \quad\rho_+ (x)= \frac{\rho_{0}}{ 1+\frac{1}{y(x)}} = y(x)\rho_-(x).
\end{equation*}

\noindent Using the above equation, we find:

\begin{align}\label{DELTAbis}
  \begin{split}
  \Delta (x) &=\rho_-(x) S_{-} (x)+\rho_+ (x)S_{+}(x) - \rho_{0} S_{0}(x) \\
  &=\rho_{0} \left\{ \left[ S_{-}(x) - S_{0}(x)\right] -  \frac{y(x)}{1+y(x)} \left[ S_{-}(x) - S_{+}(x)\right] \right\}.
  \end{split}
\end{align}

\noindent  For the quadratures $S_+ , S_-$ and  $S_{0} $, we respectively introduce the changes of variables  $z\deq (B+A) u$,  $z\deq (B-A) u$ and $z\deq B_0 u$. This together with integrations by parts enables one to write:

\begin{equation}
\label{RELSET}
\left\{
\begin{array}{l}
S_+(x) = (B+A) \int_{0}^{\infty} h(x+ u) \ee^{- (B+A) u } \dd u= \\ \qquad \qquad \qquad \qquad \qquad h(x) + \int_{0}^{\infty} \left[ \partial_{u} h(x+u)\right] \ee^{- (B+A) u} du, \\\\
S_-(x) = (B-A) \int_{0}^{\infty} h(x+ u) \ee^{- (B-A) u } \dd u= \\  \qquad \qquad \qquad \qquad \qquad h(x) + \int_{0}^{\infty} \left[ \partial_{u} h(x+u)\right] \ee^{- (B-A) u} du, \\\\
S_{0}(x) = B_0 \int_{0}^{\infty} h(x+ u) \ee^{- B_0 u } \dd u= \\ \qquad \qquad \qquad \qquad \qquad \qquad h(x) + \int_{0}^{\infty} \left[ \partial_{u} h(x+u)\right] \ee^{- B_0 u} du,
\end{array}
\right.
\end{equation}

\noindent Using Eq.(\ref{RELSET}), and the fact that $y(x) \geq 0$, one rewrites Eq.(\ref{DELTAbis}) as :

\begin{align}
\label{DELTA1}
\begin{split}
&\rho_0 \Delta (x) = \int_{0}^{\infty} \partial_{u} h(x+u)
\left\{ 
\ee^{-(B-A)u} - \ee^{-B_0u} - \frac{y(x)}{1 + y(x)} \left[ \ee^{-(B-A)u} - \ee^{-(B+A)u} \right] 
\right\} \dd u \\
&= \int_{0}^{\infty} \partial_{u} h(x+u)
\Bigl( \ee^{-(B-A)u} - \ee^{-B_0u} \Bigr) 
\left\{ 
  1   - \frac{y(x)}{1+y(x)} \biggl[ 1 +\frac{\ee^{-B_0u}- \ee^{- (B+A)u}}{\ee^{-(B-A)u} -\ee^{-B_0u} } \biggr]
\right\} \dd u.
\end{split}
\end{align}
\noindent Note that if $B_0\leq B-A$ the integrand in the first line of Eq.\eqref{DELTA1} is trivially non-positive. On the other hand, if $B+A<B_0$, then
$\ee^{-B_0u}- \ee^{- (B+A)u} \leq 0$ and looking at the last line in Eq.\eqref{DELTA1}, we see that $\Delta (x)\geq 0$, $x \in \mathbb{R}.$ 

\noindent Because the reward function is increasing, $y(x) \geq 0$ and $B-A \leq B_0$, we see that the sign of $\Delta$ is completely determined by the last term in the right-hand side of Eq.\eqref{DELTA1}. 
Let us introduce the parameter 
\begin{align}\label{paramK}
  K \deq \frac{(B+A)-B_0}{B_0 - (B-A)}.
\end{align}
The mean value theorem\footnote{As a reminder, the mean value theorem states that if $f$ is a differentiable function on $[a,b] \subset \mathbb{R}$, there exists a point $\zeta\in ]a,b[$ such that $f(b) - f(a) = f^{'}(\zeta)(b-a)$. In particular, by bounding the derivative $f^{'}(\zeta)$ from above (respectively from below), it is possible to bound from above (respectively from below) the difference $f(b) -f(a)$ by a linear function.} applied to the function $g_u(x)=\ee^{-ux}$ enables to write:
\begin{align}\label{MEANVALUETHM}
  1 +\frac{\ee^{-B_0u}- \ee^{- (B+A)u}}{\ee^{-(B-A)u} -\ee^{-B_0u}}
  = 1 + \frac{(B+A)-B_0}{B_0 - (B-A)}  \frac{u\ee^{-\zeta_1 u}}{u\ee^{-\zeta_2 u}}
  = 1 + K \frac{\ee^{-\zeta_1 u}}{\ee^{-\zeta_2 u}}, 
\end{align}
for some $B-A < \zeta_2 < B_0 < \zeta_1  < B+A.$\\
Choosing $\zeta_1=\zeta_2=B_0$ gives an upper bound on the right-hand side of Eq.\eqref{MEANVALUETHM}: 
\begin{align*}
  \begin{split}
    1 +\frac{\ee^{-B_0u}- \ee^{- (B+A)u}}{\ee^{-(B-A)u} -\ee^{-B_0u} }  
    &\leq 1+ K.
  \end{split}
\end{align*}
Plugging this inequality into Eq.\eqref{DELTA1} and using the fact that $y(x)\geq 0$, we find:
\begin{align*}
  \rho_0 \Delta (x) 
  &\geq 
  \int_0^{\infty} \partial_u h[x+u] \left[ \ee^{-(B-A) u } - \ee^{-B_0 u} \right] 
  \left[ 1 - \frac{(1+K) y(x)}{1+y(x)} \right] \dd u \\
  &= \int_0^{\infty} \partial_u h[x+u] \left[ \ee^{-(B-A) u } - \ee^{-B_0 u} \right] 
  \left[  \frac{1 - Ky(x)}{1+y(x)} \right] \dd u.
\end{align*}
Now observe that provided:
\begin{equation*}
  x_1 \deq \frac{1}{2A} \log \left[ \frac{\sigma_2^2(B-A)^2 K}{2\alpha} \right]\leq x,
\end{equation*} 
the integrand becomes strictly positive. Therefore, choosing $x_+= x_1$ we conclude. 
The proof for $x_{-}$ is similar: in that case, the first order contribution of $M^{(2)}$ comes from $S_+$ (instead of $S_-$ see Eq.\eqref{DELTAbis}) and the proof carries on, replacing $y$ by $z(x) \deq \frac{2\alpha}{\sigma_2^2(B+A)^2} \ee^{2Ax}$. 
\end{proof}

\vspace{0.3cm}

\noindent Let us conclude this paper with a few remarks. First let us consider the problem of determining $x_{\pm}$ when $h$ is explicitly given (one can think for instance of the function $h(x)=1-\mathrm{e}^{-cx}$, $c>0$). Then the exact positions of the thresholds $x_{\pm}$ can be calculated. 
In particular, at the initial position $X_0^{(1)}=X_0^{(2)}=0$, the two arms are compared via their Laplace transform. Indeed, let us denote by $\tilde{h}$ the Laplace transform of $h$ defined by:
\begin{align*}
  \tilde{h}(s) \deq \int_{0}^{\infty} h(u) \ee^{-su} \dd u, \quad s \in \mathbb{C}.
\end{align*}   
Then the difference of the GIs of $X^{(1)}$ and $X^{(2)}$ at time $t=0$ (recall the initial condition $X_0^{(1)}=X_0^{(2)}=0$) is given by:   
\begin{align*}
  \Delta (0) = \frac{2}{\sigma_2^2} 
  \left( \tilde{h} (B+A) + \tilde{h} (B-A) \right)  - \frac{2}{\sigma_1^2} \tilde{h} (B_0).
\end{align*}

\noindent Secondly, let us talk about the way the system evolves when we cross the border of the blue region. This transition happens continuously and is governed by the parameter   $K$ defined in Eq.\eqref{paramK}. Indeed, the proof of Theorem \ref{OPTALOC} can be slightly modified to obtain a lower bound on $x_+.$
The choice $\zeta_1=B-A$ and $\zeta_2=B+A$ gives a lower bound on the right-hand side of Eq.\eqref{MEANVALUETHM}:
 \begin{equation*}
   1 +\frac{\ee^{-B_0u}- \ee^{- (B+A)u}}{\ee^{-(B-A)u} -\ee^{-B_0u} }
   \geq 1+K \ee^{-2Au}.
 \end{equation*}
 Plugging this inequality into Eq.\eqref{DELTA1} and using the fact that $y(x)\geq0$  yields:
 \begin{align*}
   \rho_0 \Delta (x) \leq \int_0^{\infty} \partial_u h[x+u] \left[ \ee^{-(B-A) u } - \ee^{-B_0 u} \right] 
   \left[ \frac{1 - K y(x) \ee^{-2Au}}{1+y(x)} \right] \dd u.
 \end{align*}
 A sufficient condition for $\Delta(x)$ to be negative (i.e.\ a lower bound on $x_+$) requires an explicit form for the reward function $h$. Nevertheless, the above inequality suggests that the position of the threshold $x_+$ is determined by $\log K$. When the parameters get close to the upper bound in Eq.\eqref{DIAGRA}, that is when:
 \begin{align*}
   \frac{\sigma_2}{\sigma_1} \rightarrow \sqrt{1 + \frac{\Gamma}{\alpha}} + \sqrt{\frac{\Gamma}{\alpha}} 
   \Leftrightarrow 
   B_0 \rightarrow B+A 
   \Leftrightarrow
   K \rightarrow 0 
   \Leftrightarrow \log K \rightarrow - \infty,
 \end{align*}
 on expects that the threshold $x_+ \rightarrow -\infty$. Hence, $X^{(1)}$ is never engaged.  
 Conversely, when approaching the lower bound in Eq.\eqref{DIAGRA}, the denominator in the definition of $K$ goes to $0$ and on expects the threshold $x_+ \rightarrow +\infty$, leading to never engage $X^{(2)}$ is never favorable (see Figure \ref{FIGURE1}).

\noindent Finally, let us discuss the hypothesis $\frac{\Gamma}{\alpha}<2$.  This condition stems from the hypotheses of Theorem \ref{THM:GITTINS-DOOBS} which is itself a direct consequence of the hypotheses of Theorem \ref{KARA} (which can be traced back to the discussion in \cite[Section 3]{K}). It restricts the regimes in which we can rigorously study the two-arm bandit model. However, we believe this restriction is not a fundamental one. For the sake of brevity, we did not try to weaken the condition Eq.\eqref{KARA-CONDITION} to a local condition of the form:
\begin{align}\label{equ:KARACONDI_local}
  \exists x_+ \in \mathbb{R}, \, 0<\alpha_0\leq \alpha-m'(x)\leq \beta, 
  \, x\geq x_+.
\end{align}
We believe this can be achieved by adapting the proof in \cite{K}. Working with Eq.\eqref{equ:KARACONDI_local} would allow us to remove the restriction on $2\Gamma\leq\alpha$ since then Eq.\eqref{equ:KARACONDI_DMPS} holds for any $\Gamma, \alpha >0 $ as soon as $x$ is large enough.    

\bibliographystyle{plain}
\bibliography{bibliography}

\end{document}